\documentclass[11pt,a4paper]{article}                               
\usepackage[T1]{fontenc}                   
\usepackage[utf8]{inputenc}                 
\usepackage{graphicx}                      %
\usepackage{caption}        
\usepackage{mathrsfs}          %
\usepackage[top=.8in,bottom=1in,left=1in,right=1in]{geometry}
\usepackage{verbatim}
\usepackage{color}
\usepackage{picture}
\usepackage{amsmath,amsthm,amsfonts,amssymb}
\usepackage{pgf,tikz}
\usetikzlibrary{arrows}

\newtheorem{thm}{Theorem}[section]
\newtheorem{lem}[thm]{Lemma}

\newtheorem{prop}[thm]{Proposition}
\newtheorem{defn}{Definition}
\theoremstyle{definition}
\newtheorem{rem}[thm]{Remark}

\theoremstyle{remark}

\newcommand{\ds}{\displaystyle}

\newcommand{\abs}[1]{\left\vert#1\right\vert}

\newcommand{\R}{\mathbb{R}}

\newcommand{\W}{\mathcal W}
\newcommand{\de}{\partial}

\newcommand{\fhi}{\varphi}

\makeatletter 
\def\@makefnmark{} 
\makeatother 
\theoremstyle{plain}
\newtheorem*{thm*}{Theorem}

\title{Two estimates for the first Robin eigenvalue of the Finsler Laplacian with negative boundary parameter.} 
\author{
   Gloria Paoli%
\thanks{
Universit\`a degli studi di Napoli Federico II, Dipartimento di Matematica e Applicazioni ``R. Caccioppoli'', Via Cintia, Monte S. Angelo - 80126 Napoli, Italia. Email: gloria.paoli@unina.it}
{, }
Leonardo Trani%
\thanks{Universit\`a degli Studi di Napoli Federico II, Dipartimento di Matematica e Applicazioni ``R. Caccioppoli'', Via Cintia, Monte S. Angelo - 80126 Napoli, Italia. Email: leonardo.trani@unina.it}
}\date{}

\begin{document}

\maketitle
\begin{abstract}
\noindent We prove two bounds for the first Robin eigenvalue of the Finsler Laplacian with negative boundary parameter in the planar case. In the constant area problem, we show that the Wulff shape is the maximizer only for values which are close to 0 of the boundary parameter
and, in the fixed perimeter case, that  the Wulff shape maximizes the first eigenvalue for all values of the parameter.
\end{abstract}



\section{Introduction}
Let $\Omega$ be a bounded,  open subset of $\mathbb{R}^n$, $n\geq2$, with Lipschitz boundary; its Robin eigenvalues related to the Laplacian are the real numbers $\lambda$ such that   
\begin{equation}\label{rob}
\begin{cases}
-\Delta u=\lambda u &\mbox{in}\ \Omega\\[.2cm]
\frac{\de u}{\de \nu}+\alpha u=0&\mbox{on}\ \de\Omega
\end{cases}
\end{equation}
admits non trivial $W^{1,2}(\Omega)$ solutions. We denote  by $\partial u/\partial \nu$ the outer normal derivative to $u$ on $\partial\Omega$; $\alpha$ is an arbitrary  real constant, which will be refered to as boundary parameter of the Robin problem.  We observe that for $\alpha=0$ we obtain the Neumann problem and for $\alpha=\pm\infty$ we  formally obtain the Dirichlet problem.
For each fixed $\Omega $ and $\alpha$ there is a sequence of eigenvalues 
$$\lambda_1(\alpha,\Omega)\leq\lambda_2(\alpha,\Omega)\leq\dots\rightarrow+\infty $$
which depend on $\alpha$.
In particular, the first non trivial Robin eigenvalue of $\Omega$ is characterized by the expression
\begin{equation*}\label{var_char}
\lambda_1(\alpha , \Omega)=\min_{\substack{u\in W^{1,2}(\Omega) \\ u\neq 0}}\dfrac{\ds\int_{\Omega}\left\vert Du\right\vert^2\;dx+\alpha\ds\int_{\de\Omega}|u|^2\;d\mathcal{H}^1}{\ds\int_{\Omega}|u|^2\;dx}.
\end{equation*}

It can be proved that  this  infimum is achieved by a function $u_\alpha\in W^{1,2}(\Omega)$ and since $ \lambda_1(\alpha,\Omega)$ is simple, the corresponding eigenfunction can be chosen to be positive in $\Omega$.
We refer to \cite{k}  for a collection of the eigenvalue properties of the Robin Laplacian and the related proofs.

If we analyse the problem of minimizing the first eigenvalue of the Dirichlet problem under volume constraint, the Faber-Krahn inequality tells us that the unique solutions  are given by balls (see \cite{fa}). For the case of Neumann boundary conditions we can find analogous isoperimetric spectral inequalities in the works of Szeg\"{o} and Weinberger (\cite{sz}).

We consider now the Robin boundary conditions. If $\alpha$ is positive, we have that the  ball minimizes $\lambda_{1}(\alpha,\Omega)$  among all Lipschitz domains of given volume. This fact was proved by Bossel and Daners (\cite{d}) and generalized to the $p$-Laplacian by Dai and Fu in \cite{df}  and by Bucur and Daners in \cite{bd}; this result was also shown to hold on general open sets of finite measure by Bucur and Giacomin, see \cite{bu}. Moreover this inequality is sharp: if the first Robin eigenvalue of $\Omega$ is equal to the first eigenvalue of the ball then $\Omega$ is a ball up to a negligible set.

If $\alpha$ is negative and $\Omega\subset\mathbb{R}^n$, with $n>2$, is a bounded smooth domain, it is not true that 
\begin{equation}\label{dis}
	\lambda_1( \alpha , \Omega )\leq\lambda_1( \alpha , B ),
\end{equation}
where $B$ is a ball of the same volume as $\Omega$; a counterexample  is provided in \cite{fk}. The above fact is true within the class of Lipschitz sets which are close to a ball in a Hausdorff metric sense, see for instance \cite{fnt}.

On the other hand the spectral inequality \eqref{dis} holds in dimension $2$: in \cite{fk} is proved that for bounded planar domains of class $C^2$ and fixed  area  there exists a negative number $\alpha_*$, depending only on the area, such that \eqref{dis} holds for all $\alpha\in[\alpha_*,0]$.
This fact is proved by applying the method of parallel coordinates, introduced  by Payne and Weinberger in  \cite{pw}.

In the first part of this work we have found an analogous of inequality  \eqref{dis} in the anisotropic case. 
Let  $F$ be a Finsler norm, i.e. a convex positive $C^2$ function. We consider the anisotropic version of problem \eqref{rob}, that is 
\begin{equation*}
\begin{cases}
-{\rm div}\left(F(Du)F_{\xi}(Du)\right)=\lambda_F(\alpha , \Omega)u &\mbox{in}\ \Omega\\[.2cm]
\langle F(Du)F_{\xi}(Du),\nu_{\de\Omega}\rangle+\alpha F(\nu_{\de\Omega})u=0 &\mbox{on}\ \de\Omega.
\end{cases}
\end{equation*}
 We have the following variational characterization of the first eigenvalue:
 \begin{equation*}\label{var_char}
 \lambda_{1, F}( \alpha , \Omega )=\min_{\substack{u\in W^{1,2}(\Omega) \\ u\neq 0}}=\dfrac{\ds\int_{\Omega}F^2(Du)\;dx+\alpha\ds\int_{\de\Omega}|u|^2F(\nu_{\de\Omega})\;d\mathcal{H}^1}{\ds\int_{\Omega}|u|^2\;dx}.
 \end{equation*}
 This problem is studied for istance in \cite{dg, dg2, dpg, gt18}.
 Using the method of parallel coordinates, adapted to the anisotropic case, we prove the following theorem. 
 \begin{thm}\ For bounded planar domains of class $C^2$ and fixed area, there exists a negative number $\alpha_*$, depending only on the area, such that the following inequality  holds $\forall \alpha\in[\alpha_*,0]$:
 	$$\lambda_{1 ,F}( \alpha , \Omega )\leq\lambda_{1, F} (\alpha , \mathcal{W}^*_{\Omega}), $$
 	where $ \mathcal{W}^*_{\Omega}$ is the Wulff shape of the same area as $\Omega$.
 \end{thm}
 We recall that the Wulff shape centered at the point $x_0$ is defined as
 $$\mathcal W_r(x_0) = \{  \xi \in  \R^n \colon F^o(\xi-x_0)< r \}.$$
 In the second part of the work we generalize to the anisotropic case a result presented in \cite{afk}. Here the authors, using again the methods of parallel coordinates,  have proved that, if $\alpha<0$ and for bounded planar domains of class $C^2$, then
\begin{equation}\label{diss}
	\lambda_1( \alpha , \Omega)\leq\lambda_1(\alpha , \widetilde{B}),
	\end{equation}
where $\widetilde{B}$ is a disk with the same perimeter as $\Omega$.
We obtain the following result.
\begin{thm}
	Let $\alpha\leq 0$. For bounded planar domains of class $C^2$, we have
	$$ \lambda_{1,F}( \alpha , \Omega )\leq \lambda_{1,F}(\alpha , \widetilde{\mathcal{W}}_{\Omega} ),$$
	where $\widetilde{\mathcal{W}}_{\Omega}$ is the Wulff shape with the same perimeter as $\Omega$.
\end{thm}
In conclusion, we recall that in \cite{bcnt} the authors prove that the inequality \eqref{diss} holds true in $\mathbb{R}^n$, if it is restricted to the class of convex sets, or more precisely 
to the class of Lipschitz sets  that can be written as $\Omega\setminus K$, with $\Omega$ open and convex and $K$ closed. Moreover in \cite{fl} the authors prove that the second eigenvalue of the Robin problem related to the Laplacian is maximal for the ball among domains of fixed volume.

The paper is organized as follows. In Section $2$, we recall some basic definitions and properties of the Finsler norm $F$; in Section $3$ we state the Robin problem with negative boundary parameter in the anisotropic case. The main results  are contained in Section  $4$ and $5$,: in the first one we obtain an isoperimetric estimates with a volume constraint and in the second one with a perimeter constraint, both in dimension $2$ and in the anisotropic case.

\section{Notation and preliminaries}
In the following we denote by $\langle\cdot,\cdot\rangle$ the standard euclidean scalar product in $\mathbb{R}^n$ and by $|\cdot|$ the euclidean norm in $\mathbb{R}^n$, for $n\geq 2$.
We denote with $\mathcal{L}^n$ the Lebesgue measure in $\mathbb{R}^n$ (sometimes denoted with $V(\cdot)$) and with $\mathcal{H}^k$, for $k\in [0,n]$, the $k-$dimensional Hausdorff measure in $\mathbb{R}^n$. If $\Omega\subseteq \mathbb{R}^n$, ${\rm Lip}(\de \Omega)$ (resp. ${\rm Lip}(\de \Omega; \mathbb{R}^n)$) is the class of all Lipschitz functions (resp. vector fields) defined on $\de\Omega$.
If $\Omega$ has Lipschitz boundary, for $\mathcal{H}^{n-1}-$ almost every $x\in\partial\Omega$, we denote by $\nu_{\de\Omega}(x)$ the outward unit euclidean normal to $\partial \Omega$ at $x$ and by $T_x(\de\Omega)$ the tangent hyperplane to $\de\Omega$ at $x$.

Let $F$ be  a convex, even, 1-homogeneous and non negative  function defined in $\R^{n}$. 
Then $F$ is a convex function such that
\begin{equation}
\label{eq:omo}
F(t\xi)=|t|F(\xi), \quad t\in \R,\,\xi \in  \R^{n}, 
\end{equation}
and such that
\begin{equation}
\label{eq:lin}
a|\xi| \le F(\xi),\quad \xi \in  \R^{n},
\end{equation}
for some constant $a>0$. The hypotheses on $F$ imply that there exists $b\ge a$ such that
\begin{equation*}
\label{upb}
F(\xi)\le b |\xi|,\quad \xi \in  \R^{n}.
\end{equation*}
Moreover, throughout the paper we will assume that $F\in C^{2}(\mathbb R^{n}\setminus \{0\})$, and
\begin{equation*}
\label{strong}
[F^{p}]_{\xi\xi}(\xi)\text{ is positive definite in } \R^{n}\setminus\{0\},
\end{equation*}
with $1<p<+\infty$. 
The polar function $F^o\colon \R^n \rightarrow [0,+\infty[$ 
of $F$ is defined as
\begin{equation*}
F^o(v)=\sup_{\xi \ne 0} \frac{\langle \xi, v\rangle}{F(\xi)}. 
\end{equation*}
It is easy to verify that also $F^o$ is a convex function
which satisfies properties \eqref{eq:omo} and
\eqref{eq:lin}. $F$ and $F^o$ are usually called Finsler norm. Furthermore, 
\begin{equation*}
F(v)=\sup_{\xi \ne 0} \frac{\langle \xi, v\rangle}{F^o(\xi)}.
\end{equation*}
The above property implies the following anisotropic version of the Cauchy Shwartz inequality
\begin{equation*}
\label{imp}
|\langle \xi, \eta\rangle| \le F(\xi) F^{o}(\eta), \qquad \forall \xi, \eta \in  \R^{n}.
\end{equation*}
We can then introduce the set
$$\mathcal W = \{  \xi \in  \R^n \colon F^o(\xi)< 1 \},$$
the so-called Wulff shape centered at the origin. We put
$\kappa_n=V(\mathcal W)$.  More generally, we denote by $\mathcal W_r(x_0)$
the set $r\mathcal W+x_0$, that is the Wulff shape centered at $x_0$
with measure $\kappa_nr^n$, and $\mathcal W_r(0)=\mathcal W_r$. In particular, when $\W$ is a subset of $\R^2$, we write $\abs{\W}=\kappa$.


We conclude this paragraph reporting the following properties of $F$ and $F^o$:
\begin{gather*}
\label{prima}
\langle \nabla_\xi F(\xi) , \xi \rangle= F(\xi), \quad  \langle\nabla_\xi F^{o} (\xi), \xi \rangle
= F^{o}(\xi),\qquad \forall \xi \in
\R^n\setminus \{0\}
\\
\label{seconda} F(  \nabla_\xi F^o(\xi))=F^o( \nabla_\xi F(\xi))=1,\quad \forall \xi \in
\R^n\setminus \{0\}, 
\\
\label{terza} 
F^o(\xi)   \nabla_\xi F( \nabla_\xi F^o(\xi) ) = F(\xi) 
\nabla_\xi F^o\left(  \nabla_\xi F(\xi) \right) = \xi\qquad \forall \xi \in
\R^n\setminus \{0\}. 
\end{gather*}


We  recall now some basic definitions and theorems concerning the perimeter in the Finsler norm.
\begin{defn}
	Let $\Omega$ be a bounded open subset of $\R^2$ with Lipschitz boundary, the anisotropic perimeter of $\Omega$ is defined as 
	\[
	P_F(\Omega)=\displaystyle \int_{\de \Omega}F(\nu_{\de \Omega}) \, d \mathcal H^{1}.
	\]
\end{defn}
Clearly, the anisotropic perimeter of  $\Omega$ is finite if and only if the usual Euclidean perimeter of $\Omega$, that we denote by $P(\Omega)$, is finite. Indeed, by the quoted properties of $F$, we obtain that
$$
aP(\Omega) \le P_F(\Omega) \le bP(\Omega).
$$
For example, if $\Omega=\mathcal{W}_R$, then 
$$P_F(\mathcal{W_R})=2\kappa R.$$
\noindent Moreover, an isoperimetric inequality is proved for the anisotropic perimeter, see for istance \cite{aflt,bu,dp,dpg,fm}.
\begin{thm*}
	Let $\Omega$ be a subset of $\mathbb{R}^2$ with finite perimeter. Then 
	\begin{equation}\label{anis_iso_inequality}
		P_F(\Omega)^2 \ge 4 \kappa V(\Omega )
	\end{equation}
	and the equality holds if and only if $\Omega$ is homothetic to a Wulff shape.
\end{thm*}
Moreover, if $K$ is a bounded convex subset of $\mathbb{R}^2$, and $\delta>0$, the following Steiner formulas hold (see \cite{a,s}):
\begin{equation}\label{steiner_1}
	V(K+\delta\mathcal{W})=V(K)+P_F(K)\delta+\kappa\delta^2;
\end{equation}
\begin{equation}\label{steiner_2}
	P_F(K+\delta\mathcal{W})=P_F(K)+2\kappa\delta.
\end{equation}
Let $\Omega $ be a bounded open set of $\mathbb{R}^2$, the anisotropic distance of a point $x\in\Omega$ to the boundary $\de\Omega$ is defined as
$$ d_F(x,\de\Omega)=\inf_{y\in\de\Omega}F^o(x-y).$$ 
By the properties of the Finsler norm $F$, the distance function satisfies
\begin{equation}
	F(Dd_F(x))=1\quad \mbox{a.e. in }\Omega
\end{equation}
For the properties of the anisotropic distance function we refer, for istance, to \cite{cm}.
We can define also the anisotropic inradius of $\Omega$ as 
$$r_F(\Omega)=\sup\{d_F(x,\de\Omega),\; x\in\Omega\} .$$

\noindent We denote by 
$$\tilde{\Omega }_t=\{x\in\Omega\;|\;d_F(x,\de\Omega)>t \},$$
with $t\in[0,r_F(\Omega)]$. The general Brunn-Minkowski theorem (see \cite{s}) and the concavity of the anisotropic distance function give that the function $P_F(\tilde{\Omega}_t)$ is concave in $[0,r_F(\Omega)]$, hence it is decreasing and absolutely continous. In \cite{dpg} the following result is stated.
\begin{lem}\label{der}
	For almost every $t\in(0,r_F(\Omega))$,
	$$-\dfrac{d}{dt}V\left(\tilde{\Omega}_t\right)=P_F(\tilde{\Omega}_t). $$
	
\end{lem}

\section{The Robin problem in the anisotropic case}
Let $\Omega$ be a bounded subset of $\mathbb{R}^2$ of class $C^2$.
We consider the anisotropic eigenvalue problem with Robin boundary conditions.

\noindent We fix a negative number $\alpha$ and we study the following problem:
\begin{equation}\label{min_problem}
\lambda_{1,F}(\alpha , \Omega)=\min_{\substack{u\in W^{1,2}(\Omega) \\ u\neq 0}} J(u),
\end{equation}
where 
\begin{equation}\label{var_char}
J(u)=\dfrac{\ds\int_{\Omega}\left(F(Du)\right)^2\;dx+\alpha\ds\int_{\de\Omega}|u|^2F(\nu_{\de\Omega})\;d\mathcal{H}^1}{\ds\int_{\Omega}|u|^2\;dx},
\end{equation}
and $\nu_{\de \Omega}$ is the outer normal to $\de\Omega$.  Using a constant as test function, we obtain the following inequality
\begin{equation}\label{in.lam}
\lambda_{1,F}(\alpha , \Omega)\le \alpha\frac{P_F(\Omega)}{\abs{\Omega}}\le 0.
\end{equation}
The minimizers $u$ of problem \eqref{min_problem} satisfy the following eigenvalue
\begin{equation*}
\begin{cases}
-{\rm div}\left(F(Du)F_{\xi}(Du)\right)=\lambda_{1,F}(\alpha , \Omega)u &\mbox{in}\ \Omega\\[.2cm]
\langle F(Du)F_{\xi}(Du),\nu_{\de\Omega}\rangle+\alpha F(\nu_{\de\Omega})u=0 &\mbox{on}\ \de\Omega
\end{cases}
\end{equation*}
that is, in the weak sense
\begin{equation}\label{wf}
\ds\int_{\Omega} F(Du)\left\langle D_\xi F(Du),D\fhi\right\rangle\; dx + \alpha\ds\int_{\de\Omega} u\fhi F(\nu_{\de\Omega})\; d\mathcal{H}^1=\lambda_{1,F}(\alpha , \Omega)\ds\int_{\Omega}u\fhi\; dx,
\end{equation}
for all $\fhi \in W^{1,2}(\Omega)$.
The following proposition is proved in \cite{dg2}.
\begin{prop}
There exist a function $u\in C^{1,\alpha}(\Omega)\cap C(\bar{\Omega})$ which realizes the minimum in \eqref{min_problem}and satisfies the anisotropic Robin problem. Moreover, $\lambda_{1,F}(\alpha,\Omega)$ is the first eigenvalue of the Robin problem and the first eigenfunctions are positive (or negative) in $\Omega$.
\end{prop}
\section{Isoperimetric estimates with a volume constraint}

In the following we are fixing a Finsler norm $F$. 

\begin{thm}\label{iso_vol} For bounded planar domains of class $C^2$ and fixed area, there exists a negative number $\alpha_*$, depending only on the area, such that the following inequality  holds $\forall \alpha\in[\alpha_*,0]$:
	$$\lambda_{1,F}(\alpha , \Omega)\leq\lambda_{1,F}(\alpha , \mathcal{W}^*_{\Omega}), $$
	where $\mathcal{W}^*_{\Omega}$ is the Wulff shape of the same area as $\Omega$. 
\end{thm}
 In order to prove Theorem \ref{iso_vol} we adapt in the anisotropic case the proof of Freitas and Krejcirik contained in \cite{fk}. This proof makes use of the classical method of parallel coordinates, developed for the Euclidean case in \cite{pw} and for the Riemanian case in \cite{s}.

We assume that $\de\Omega$ is composed by a finite union of $C^2$ Jordan curves $\Gamma_0,\dots,\Gamma_N$, where $\Gamma_0$ is the outer boundary of $\Omega$, i.e. $\Omega$ lies in the interior $\Omega_0$ of $\Gamma_0$. We observe that, if $N=0$, then $\Omega$ is simply connected and $\Omega=\Omega_0$. We denote by $$L_0^F:=P_F(\Omega_0)$$
the outer anisotropic perimeter.  Therefore, by the anisotropic isoperimetric inequality, we have
\begin{equation}
	(L_0^F)^2\geq 4\kappa A_0,
	\end{equation}
where $A_0=V(\Omega)$ denotes the area of $\Omega$ (not of $\Omega_0$).

\noindent We now introduce the \textbf{anisotropic parallel coordinate method} based at the outer boundary $\Gamma_0$. Let $\rho_F:\Omega_0\rightarrow(0,\infty)$ be the anisotropic distance function from the outer boundary $\Gamma_0$:
$$ \rho_F(x)=d_F(x,\Gamma_0).$$
\noindent Let $$A_F(t)=V(\{x\in\Omega\;|\;0<\rho_f(x)<t\}) $$ 
denote the area of $\Omega_t=\Omega\setminus\tilde{\Omega}_t$.
and let
$$L_F(t)=\int_{\rho_F^{-1}(t)\cap \Omega}  F(\nu_{\de\Omega}(x))\;d\mathcal{H}^1(x).$$

\begin{rem}
	By lemma \ref{der}, we obtain that, for almost every $t\in[0,r_F(\Omega_0)]$,
	\begin{equation}\label{area_len}
	A'_F(t)=L_F(t).
	\end{equation}
\end{rem}

\subsection{Step $1$: use of the anisotropic parallel coordinates.}    

Let $\phi:[0,r_F(\Omega)]\rightarrow\mathbb{R}$ be a smooth function and consider the test function $$ u=\phi\circ A_F\circ \rho_F,$$ which is Lipschitz in $\Omega$. Using the anisotropic parallel coordinates, the coarea formula and the fact that $F\left(D\rho_F\right)=1$, we obtain the following relations:

	\begin{align*}
 &\quad	||u||^2_{L^2(\Omega)}=\int_{\Omega}  u^2(x)\;dx=\int_{\Omega} \left(\phi\circ A_F\circ\rho_F(x)\right)^2dx=\\ =& \int_{0}^{r_F(\Omega)}\left(\int_{\{\rho_F(x)=t \}} \left(\phi\circ A_F\circ\rho_F(x)\right)^2\dfrac{1}{|D\rho_F(x)|}\;d\mathcal{H}^1(x)\right)\;dt\\& =\int_{0}^{r_F(\Omega)} \phi(A_F(t))^2\;P_F(  \{  \rho_F(x)<t  \}   )\;dt=\\&=\int_{0}^{r_F(\Omega)} \phi(A_F(t))^2\;A'_F(t)\;dt;
\\&\\
	\\&\quad\int_{\Omega}\left( F^2\left(Du(x)  \right) \right)dx=\int_{\Omega} F^2\left(  \phi'\left(A_F\circ\rho_F\left(x\right)\right)A'_F\left(\rho_F\left(x\right)\right) D \rho_F\left(x\right) \right)dx=\\&=\int_{\Omega} \left(\phi'\left(A_F\circ\rho_F\left(x\right)\right)\right)^2\left( A'_F\left(\rho_F\left(x\right)\right)\right)^2dx=\int_{0}^{r_F(\Omega)}\left(\phi'\left(A_F\left(t\right)\right)\right)^2\left(A'_F\left(t\right)\right)^3dt;\\&\\
	\\&\int_{\de\Omega} |u(x)| ^2F(\nu_{\de\Omega}(x))\;d\mathcal{H}^1(x)=\int_{\de\Omega}\left(\phi\circ A_F\circ\rho_F\left(x\right)\right)^2 F(\nu_{\de\Omega}(x))\;d\mathcal{H}^1(x)=\\[-0.5cm]
	\\&=\left(\phi\circ A_F\left(0\right)\right)^2 P_F(\Omega)\geq \phi^2(0) \;L_0.
	\end{align*}
	$ $ \\ 
	$ $ \\ 
	Therefore we have  that
	\begin{equation}\label{prv}
	 \lambda(\Omega)\leq \dfrac{\int_{0}^{r_F(\Omega)}\left(\phi'\left(A_F\left(t\right)\right)\right)^2\left(A'_F\left(t\right)\right)^3dt+\alpha\; \phi^2(0) \;L_0^F}{\int_{0}^{r_F(\Omega)} \phi(A_F(t))^2\;A'_F(t)\;dt}.
	\end{equation}

	


\subsection{Step $2$: from domains to annuli.}
 We adapt in the anisotropic case the idea contained in \cite{pw}. We consider the following change of variables:
 \begin{equation}\label{change_variables}
 R(t):=\dfrac{\sqrt{\left(L^F_0\right)^2-4\kappa A_F(t)}}{2\kappa}
 \end{equation}
on the interval $[r_1,r_2]$, where 
\begin{equation}\label{radii}
	r_1:=R\left(r_F\left(\Omega\right)\right)=\dfrac{\sqrt{\left(L^F_0\right)^2-4\kappa A_0}}{2\kappa},\qquad r_2:=R(0)=\dfrac{L_0^F}{2\kappa}.
\end{equation}

\begin{rem}
	Thanks to \eqref{anis_iso_inequality}, the transformation  \eqref{change_variables} is well defined on the set  $[0, r_F(\Omega)]$.
\end{rem}
We introduce now the function 
$$\psi(r):=\phi\left(\dfrac{\left(L_0^F\right)^2}{4\kappa}-\kappa r^2 \right)$$
and we obtain the following expressions:
$$\int_{\Omega}  u^2(x)\;dx= 2\kappa\int_{r_1}^{r_2} \left(\psi(r)\right)^2 r\;dr;$$
$$ \int_{\Omega}\left( F^2\left(Du(x)  \right) \right)dx= 2\kappa\int_{r_1}^{r_2}\left(\psi'(r)\right)^2\left(R'(r)\right)^2 r\;dr ;$$
$$ \int_{\Omega} |u(x)| ^2F(\nu_{\de\Omega}(x))\;dx\geq L^F_0\; \psi(r_2)^2.$$

\begin{rem}
	The radii in \eqref{radii} are such that the $F$-annulus $A^F_{r_1,r_2}:=\mathcal{W}_{r_2}\setminus\overline{\mathcal{W}}_{r_1}$ has the same area $A_0$ as the original domain $\Omega$.
	We observe that the  transformation \eqref{change_variables} maps $\de\Omega_t$ into the Wulff shape of radius $R(t)$; so $\Gamma_0$ is mapped into the Wulff shape of equal anisotropic perimeter. Moreover, $\Omega_t$ is mapped in the anisotropic annulus of area $A_F(t)$.
\end{rem}

\begin{prop}
	Let $\Omega$ be a bounded planar domain of class $C^2$, then
 \begin{equation*}
 |R'(t)|\leq 1,
 \end{equation*}
 where $R$ is defined in \eqref{change_variables}.
\end{prop}

\begin{proof}
	From \eqref{area_len} follows that, for almost every $t\in [0,r_F(\Omega)]$ we have 
\begin{equation}\label{rad}
R'(t)=-\dfrac{L_F(t)}{\sqrt{\left(L_0^F\right)^2-4\kappa A_F(t)}}.
\end{equation}
Using the Steiner formula we obtain for almost every $t\in[0,r_F(\Omega)]$
$$  L_F(t)\leq L_0^F-2\kappa t;$$
$$A_F(t)=\int_{0}^{t}L_F(v)\;dv\leq L_0^F t-\kappa t^2. $$
Therefore,
$$ L_F(t)^2\leq \left(L_0^F\right)^2-4\kappa A_F(t),$$
and putting this in \eqref{rad} the thesis follows.

\end{proof}

We obtain this upper bound
\begin{equation}\label{bound}
	\lambda_{1,F}(\alpha , \Omega)\leq \inf_{\psi\neq 0} \dfrac{\int_{r_1}^{r_2}\psi'(r)^2 r\;dr+\alpha\;r_2\:\psi(r_2)^2}{\int_{r_1}^{r_2}\psi(r)^2r\;dr}:=\mu(\alpha , A^F_{r_1,r_2}),
\end{equation}
so the infimum is attained for the first eigenfunction of the Laplacian in $A^F_{r_1,r_2}$, with anisotropic Robin boundary condition on $\de\mathcal{W}_2$ and anisotropic Neumann boundary conditions on $\de\mathcal{W}_1$.
Therefore we have proved the following proposition.
\begin{prop}\label{annulus}
Let $\alpha\leq 0$. For any bounded planar domain $\Omega$ of class $C^2$,
$$\lambda_{1,F}( \alpha ; \Omega)\leq \mu(\alpha , A^F_{r_1,r_2}) ,$$
where $A^F_{r_1,r_2}$ is the anisotropic annulus of the same area as $\Omega$ with radii \eqref{radii}.
\end{prop}

\subsection{Step $3$: from annuli to disks.}
Let $\mathcal{W}_{r_1,r_2}$ be the Wulff shape of the same area as the anisotropic annulus $A^F_{r_1,r_2}$, which has the same area $A_0$ as $\Omega$. So, we have that
\begin{equation}\label{r3}
	r_3=\sqrt{\dfrac{A_0}{\kappa}},
\end{equation}
where $r_3$ is the radius of $\mathcal{W}_{r_1,r_2}$.
In \cite{fk} we find the following asymptotics as $\alpha\rightarrow+\infty$:
\begin{equation}\label{asy1}
	 \lambda_{1,F}( \alpha , \mathcal{W}_{r_1,r_2})= 2\alpha \dfrac{r_3}{r_3^2}+O(\alpha^2) \quad \text{(Robin Wulff);}
\end{equation}
\begin{equation}\label{asy2}
\mu^\alpha(A^F_{r_1,r_2})=2\alpha\dfrac{r_2}{r_3^2}+O(\alpha^2) \quad \text{ (Neumann-Robin annulus). }
\end{equation}
 Using them we can prove that, for $\alpha<0$ small enough, 
 \begin{equation}\label{disk_annulus}
 	\mu(\alpha , A^F_{r_1,r_2})\leq \lambda_{1,F}(\alpha , \mathcal{W}_{r_1,r_2}),
 \end{equation}
 where $\mathcal{W}_{r_1,r_2}$ is the Wulff shape of the same area as the anisotropic annulus $A^F_{r_1,r_2}$.
 Thus, we have proved the following theorem.
 \begin{prop}
 	For any bounded domain $\Omega$ of class $C^2$, there exists a negative number $\alpha_0=\alpha_0(A_0, L_0^F)$ such that $$\lambda_{1,F}(\alpha , \Omega)\leq\lambda_{1,F}( \alpha , \mathcal{W}^*_{\Omega}) $$
 	holds $\forall\alpha\in[\alpha_0,0]$, where $\mathcal{W}^*_{\Omega}$ is the Wulff shape of the same area as $\Omega$.
 \end{prop}
  
  \begin{rem}
  	Using the above asymptotics  we can show that
  	\begin{equation*}
  		\dfrac{d}{d\alpha}\lambda_{1,F}( \alpha , \Omega)\arrowvert_{\alpha=0}=\dfrac{\mathcal{H}^1(\de\Omega)}{|\Omega|}.
  	\end{equation*}
  \end{rem}
 
\subsection{Step $4$:  uniform behaviour and conclusion.}

In order to complete the proof of the Theorem \ref{iso_vol}, it remains only to show the following fact.
\begin{prop}\label{pro}
The constant $\alpha_0$ of Proposition \ref{annulus} is indipendent of $L_0$.
\end{prop}

 Following \cite{fk}, we need to show that the neighbourhood of zero in which \eqref{disk_annulus} does not degenerate in both cases  when $r_1\rightarrow 0$ and $r_2\rightarrow +\infty$, 
So, we are going to prove that  $\alpha_0$ remains bounded away from $0$ uniformly in this two istances.

We fix $\epsilon>0$ and we consider $$ r_1=\sqrt{(2\epsilon r_3+\epsilon^2)}, \qquad r_2=r_3+\epsilon,$$
where $r_3$ is fixed and equall to $\sqrt{A_0/\kappa}$. In an analogous way to the one reported in \cite{fk}, it can be proved that there exists $\alpha^*<0$ such that the curve $\Gamma_A:\alpha\longmapsto\mu^\alpha(A^F_{r_1,r_2})$ stays below the curve $\Gamma_B:\alpha \longmapsto\lambda_{1,F}(\alpha , \mathcal{W}_{r_3})$ for all $\epsilon>0$ and $\forall \alpha\in(\alpha^*,0)$.                                                  

Because of the simplicity of the eigenvalues, both the curves are analytic. Moreover, taking into account the asymptotics \eqref{asy1} and \eqref{asy2} we have that
$$\dfrac{d}{d\alpha} \mu(\alpha , \mathcal{W}_{r_1,r_2})\leq \dfrac{d}{d\alpha} \lambda_{1,\alpha}(\alpha , A^F_{r_1,r_2}).$$

\begin{rem}
We prove that the curves $\Gamma_A$ are concave in $\alpha$. Let $\epsilon>0$ and let  $\psi$ be the first eigenfunction $\mu^{\alpha+\epsilon}(A^F_{r_1,r_2})$ of the Laplacian in the anisotropic annulus: We can choose $\psi$ normalised in to $1$, so we have
\begin{equation}\label{eig1}
\mu^{\alpha+\epsilon}(A^F_{r_1,r_2})=\int_{r_1}^{r_2}\psi'(r)^2 r\;dr+(\alpha+\epsilon)\;r_2\:\psi(r_2)^2.
\end{equation}
Let $\varphi$ be the first eigenfunction $\mu^{\alpha}(A^F_{r_1,r_2})$ normalized to $1$:
\begin{equation}\label{eig2}
\mu^{\alpha}(A^F_{r_1,r_2})=\int_{r_1}^{r_2}\phi'(r)^2 r\;dr+(\alpha)\;r_2\:\phi(r_2)^2.
\end{equation}
Now, putting $\phi$ as a test function in the variational formula of  $\mu^{\alpha+\epsilon}(A^F_{r_1,r_2})$ we obtain
$$\mu^{\alpha+\epsilon}(A^F_{r_1,r_2})\leq  \int_{r_1}^{r_2}\phi'(r)^2 r\;dr+(\alpha+\epsilon)\;r_2\:\phi(r_2)^2=\mu^{\alpha}(A^F_{r_1,r_2})+\epsilon  \;r_2\:\phi(r_2)^2.$$
In order to prove our claim, we need only to show that 
$$ \dfrac{d}{d\alpha}\mu^{\alpha}(A^F_{r_1,r_2})=  \;r_2\:\phi(r_2)^2.$$
\end{rem}
We prove the following more general result.
\begin{lem}
Let $\Omega$ be a bounded subset of $\mathbb{R}^2$ and let $u_\alpha$ an eigenfunction related to the eigenvalue $\lambda( \alpha ,\Omega) $, defined in \eqref{min_problem}, such that $\Vert u_\alpha\Vert_{L^2(\Omega)}=1$. Then 
\begin{equation}\label{der_lam}
\lambda_{1,F}^\prime( \alpha , \Omega ) := \frac{d\lambda_{1,F}( \alpha , \Omega )}{d\alpha}=\ds\int_{\de\Omega} u_\alpha^2 F(\nu_{\partial\Omega})d\mathcal{H}^1.
\end{equation}
\end{lem}
\begin{proof}
From the variational characterization \eqref{min_problem} and using the fact that $\Vert u_\alpha\Vert_{L^2(\Omega)}=1$ we have
\begin{equation}\label{varcharnorm}
\lambda_{1,F}( \alpha , \Omega ) = \ds\int_{\Omega} F^2(D u_\alpha)\ dx + \alpha\ds\int_{\de\Omega} u_\alpha^2 F(\nu_{\de\Omega})\ d\mathcal{H}^1.
\end{equation}	
Deriving both sides of \eqref{varcharnorm} with respect to $\alpha$, we obtain
\begin{equation}\label{der1}
\lambda_{1,F}^\prime ( \alpha , \Omega ) = 2\ds\int_{\Omega} F(D u_\alpha)D_\xi F(D u_\alpha) D u_\alpha^\prime \ dx + \ds\int_{\de\Omega} u_\alpha^2 F(\nu_{\de\Omega})\ d\mathcal{H}^1 + 2\alpha\ds\int_{\de\Omega} u_\alpha u_\alpha^\prime F(\nu_{\de\Omega})\ d\mathcal{H}^1.
\end{equation}
Using the weak formulation \eqref{wf} of the problem in the equation \eqref{der1}, remembering that $u_\alpha^\prime$ is the derivative with respect to $\alpha$ and it is in the set of the test functions by standard elliptic regularity theory, we obtain
\begin{equation}\label{awf}
\lambda_1^\prime ( \alpha , \Omega ) = 2\lambda_1( \alpha , \Omega ) \ds\int_{\Omega} u_\alpha u_\alpha^\prime\ dx + \ds\int_{\de\Omega} u_\alpha^2 F(\nu_{\de\Omega})\ d\mathcal{H}^1,
\end{equation}
and, having in mind that, from the condition $\Vert u_\alpha\Vert_{L^2(\Omega)}=1$,
\begin{equation*}
\ds\int_{\Omega} u_\alpha u_\alpha^\prime\ dx = 0
\end{equation*} 
we get, from \eqref{awf}, the equation \eqref{der_lam}.
\end{proof}
Therefore, since the $\Gamma_A$ are concave in $\alpha $ and their derivative with respect to $\alpha$ are increasing with $\epsilon$,  we have that the tangent to the curve corresponding to a specific anisotropic annulus  intersects $\Gamma_B$  at one and only one point , $\alpha_1$, to the left of zero. Thanks to the concavity we can say that, for larger value of $\epsilon$,  any $\Gamma_A$ that intersects $\Gamma_ B$ must do so to the left of $\alpha_1$.

As far as the case when $\epsilon$ is small, we follow closely the proof presented in \cite{fk}. We study the intersection points of the two curves $\Gamma_A$ and $\Gamma_B$, comparing the following two equations; the first equation is the equation of the Wulff shape 
\begin{equation}\label{wulff_bessel}
	kI_1(k r_3)+\alpha I_0(k r_3)=0;
\end{equation}
the second equation is the one of the Neumann-Robin anisotropic annulus
\begin{align*}\label{annulus_bessel}
	& K_1(k\sqrt{2\epsilon r_3+\epsilon^2}) \left[kI_1\left( k\left(r_3+\epsilon\right)\right)+\alpha I_0\left( k\left( r_3+\epsilon\right) \right)\right]-	\\& I_1(k\sqrt{2\epsilon r_3+\epsilon^2}) \left[k K_1\left( k\left(r_3+\epsilon\right)\right)-\alpha K_0\left( k\left( r_3+\epsilon\right) \right)\right]=0.
\end{align*}

We denote here with $I_\nu$ and $K_\nu$ the modified Bessel functions (for their properties we refer to \cite{abra}).
 The solution in $\alpha$ of the intersection is given by 
 $$ \alpha=-k \dfrac{I_1(kr_3)}{I_0(k r_3)  } .$$
 The proof that there are no intersections between $\Gamma_A$ and $\Gamma_B$ for $\alpha$ close to zero  is the same as the one presented in \cite{fk}. In this way we have proved Proposition \ref{pro}.
\section{Isoperimetric estimates with a perimeter constraint}
Using the method of parallel coordinates we are able to prove also the following theorem.
\begin{thm}\label{per}
Let $\alpha\leq 0$ and let $\Omega\subseteq\mathbb{R}^2$ a bounded  domain of class $C^2$. Then
$$ \lambda_{1,F}( \alpha , \Omega)\leq \lambda_{1,F}(\alpha , \widetilde{\mathcal{W}}_{\Omega}),$$
where $\widetilde{\mathcal{W}}_{\Omega}$ is the Wulff shape with the same perimeter as $\Omega$.
\end{thm}
The crucial step in order to prove this theorem is given by the following proposition.
\begin{prop}\label{ann<wulff_p}
	Let $\alpha<0$. For any $0<r_1<r_2$ we have
	$$ \mu(\alpha , A_{r_1,r_2})\leq\lambda_{1,F}( \alpha , \W_{r_2} ).$$ 
\end{prop}

\begin{proof}
	By symmetry, $\lambda_{1,F}( \alpha , \W_{r_2})$ is the smallest eigenvalue of the following one-dimensional problem 
	
	\begin{equation}\label{eq_dif}
	\begin{cases}
-r^{-(d-1)}&[r^{d-1}\phi'(r)]'=\lambda_{1,F}( \alpha , \W_{r_2} )\;\phi(r), r\in[0,r_2]\\
	&\phi'(0)=0 \\
	&\phi'(r_2)+\alpha \phi(r_2)=0.
	\end{cases}
	\end{equation}
	We can choose the associated function $\phi_1$ to be positive and normalised to $1$ and this eigenfunction can be used as a test function. Integrating by parts, we obtain
	\begin{equation}\label{mu.lam-} 
	\mu(\alpha , A_{r_1,r_2})\leq \lambda_{1,F}( \alpha , \W_{r_2} )-r_1\phi(r_1)\phi'(r_1).
	\end{equation}
	Since $\phi_1$ satisfies \eqref{eq_dif}, we have for all $r\in[0,r_2]$
	$$\left[r\phi_1(r)\phi_1'(r) \right]^\prime =-\lambda_{1,F}(\alpha , \W_{r_2})r\phi_1(r)^2+r\phi_1'(r)^2\geq0.$$
	and the inequality is due to \eqref{in.lam}. From the above inequality the function $g(r):=r\phi(r)\phi^{\prime}(r)$ is non-decreasing and using \eqref{mu.lam-}, we obtain the desired result.
\end{proof}

\begin{rem}\label{monotonicity}
	The following monotonicity result holds true. Let be $\mathcal{W}_R$ be a Wulff shape of radius $R$.  If $\alpha<0$, then 
	\begin{equation*}
	R\mapsto\lambda_{1,F}(\alpha,\mathcal{W}_R)
	\end{equation*}
	is strictly increasing. The above result is proven for the disks in \cite{afk} and for the annuli in \cite{lt}. 
\end{rem}
 
\begin{proof}[Proof of Theorem \ref{per}]
	
Firstly, we observe that the measure of $\mathcal{W}_{r_2}$ is greater than the measure of $A^F_{r_1,r_2}$ and the perimeter of $\mathcal{W}_{r_2}$, which is equal to $L_0$ is less than the perimeter of $A^F_{r_1,r_2}$. Using theorem \ref{annulus} and proposition \ref{ann<wulff_p} we obtain the thesis  for simply connected domains, i .e. when $L_0=P_F(\Omega)$. 

Concerning the general case, when there are multiple connected domains, thanks to remark \ref{monotonicity}, we have that 
\begin{equation*}
	\lambda_{1,F}( \alpha , \mathcal{W}_{r_2})\leq \lambda_{1,F}( \alpha , \mathcal{W}_{r_3}),
\end{equation*}
where $r_3=P_F(\Omega)/2\kappa$ for all $\alpha\leq 0$.
\end{proof}

\end{document}